\documentclass[12pt]{amsart}
\usepackage{epsfig,color}
\usepackage{blindtext}
\usepackage{hyperref}
\usepackage{graphicx}
\usepackage{enumitem} 
\usepackage{url}
\usepackage{amssymb}
\usepackage{graphicx,import}
\usepackage{comment}
\usepackage{esint}

\theoremstyle{definition}

\newtheorem{theorem}{Theorem}[section]

\newtheorem{conjecture}[theorem]{Conjecture}
\newtheorem{proposition}[theorem]{Proposition}

\newtheorem{remark}[theorem]{Remark}
\newtheorem{corollary}[theorem]{Corollary}

\newtheorem*{remark*}{Remark}

\numberwithin{equation}{section}

\newcommand{\R}{\mathbb{R}}

\title[High Codimension Ancient Mean Curvature Flow]{Construction of High Codimension Ancient Mean Curvature Flows}

\author{Douglas Stryker}
\address{Department of Mathematics, Massachusetts Institute of Technology, Cambridge, MA 02139, USA}
\email{stryker@mit.edu}

\author{Ao Sun}
\address{Department of Mathematics, Massachusetts Institute of Technology, Cambridge, MA 02139, USA}
\email{aosun@mit.edu}

\begin{document}

\maketitle

\begin{abstract}
We construct a class of compact ancient solutions to the mean curvature flow in Euclidean space with high codimension. In particular, we construct higher codimensional ancient curve shortening flows. Moreover, we characterize the asymptotic behavior of these solutions.

Add on remark: the construction in this paper has been discovered by Altschuler D.-Altschuler S.-Angenent-Wu in \cite{AAAW}. 
\end{abstract}

\section{Introduction}
A family of immersed $n$-dimensional submanifolds $M_t^n \subset \R^N$ evolves along the mean curvature flow if its coordinates satisfy the equation
\[
\partial_t x = -\vec{H},
\]
where $\vec{H}$ is the mean curvature vector, given by minus the trace of the second fundamental form. This equation can also be written in the useful form
\[
\partial_t x = \Delta_{M_t} x.
\]
The mean curvature flow is the negative gradient flow for the volume of submanifolds induced by $\R^N$, so solutions to this flow optimally decrease their volume.
In particular, the one dimensional mean curvature flow, called the curve shortening flow, optimally decreases the length of immersed curves.

The mean curvature flow has been the subject of extensive study in codimension one (see \cite{bwhite02}, \cite{CM-generic}, \cite{CMP-MCF}, \cite{CM-uniqueness}); namely for hypersurfaces $M_t^n \subset \R^{n+1}$. The mean curvature flow for higher codimension in Euclidean space, which is the focus of this paper, presents a unique challenge (see \cite{Wang-Long-time}, \cite{Smoczyk}, \cite{CM-complexity}, \cite[\S 0.3]{CM-regularity}). For example, the avoidance principle used to handle flows with codimension one fails for higher codimension.

In this paper, the primary objects of study are \emph{ancient} solutions to the mean curvature flow. A solution is called ancient if it is defined for all time in the interval $(-\infty, 0)$. Our main result is a constructive proof of the following theorem.

\begin{theorem}\label{thm:main}
For every positive integer $m$, there is a compact, connected, ancient solution to the curve shortening flow in $\R^{2m}$ that does not lie in any $(2m-1)$-dimensional Euclidean subspace.
\end{theorem}

As far as we know, this is the first construction of a curve shortening flow in Euclidean space that is not contained in a two-dimensional Euclidean subspace. Moreover, by taking products of these ancient curve shortening flows, we construct ancient mean curvature flows of arbitrary dimension. In particular, we prove the following theorem.

\begin{theorem}\label{thm:compact}
Let $m$ and $n$ be positive integers with $m \geq n$. There exists an $n$-dimensional, compact, connected, ancient solution to the mean curvature flow in $\R^{2m}$ that does not lie in any $(2m-1)$-dimensional Euclidean subspace.
\end{theorem}

Ancient solutions to the mean curvature flow are models for the singularities formed under the flow. Therefore, the study of ancient mean curvature flow is important to the study of the singular behavior of the mean curvature flow. Ancient solutions in codimension $1$ have been extensively studied, see \cite{Huisken-Sinestrari-ancient}, \cite{Haslhofer-Hershkovits}, \cite{ADS} \cite{Brendle-Choi}, \cite{choi2018ancient}. 

The lack of examples of higher codimensional ancient mean curvature flows is the main challenge. Presently, we know very few constructions for higher codimensional ancient mean curvature flow. Most of them are solitons of the Lagrangian mean curvature flow (see \cite{Lee-Wang}, \cite{Castro-Lerma}). Choi-Mantoulidis construct ancient mean curvature flows from unstable minimal submanifolds in \cite{ChMa}. Our construction differs from these examples in important ways.

Let us stress the difference between previous constructions and our construction by considering the ancient curve shortening flow. The solitons (including translating solitons, rotating solitons, and self-shrinkers) to the curve shortening flow are known to satisfy a second order ODE, and are therefore planar. Some other known non-soliton curve shortening flows are also planar (see \cite{DHS}). The construction by Choi-Mantoulidis requires an unstable minimal submanifold, but it is known that the only $1$-dimensional minimal submanifolds in $\R^N$ are straight lines, which are stable. Therefore our examples capture some unique features.

\bigskip
We mention two implications of our examples.

First, our example suggests the sharp value of a codimension bound of Colding-Minicozzi \cite{CM-complexity}. Recall that the entropy of a submanifold $M \subset \R^N$is defined as
\[ \lambda(M) := \sup_{s \in \R_{>0},~y \in \R^N} (4\pi)^{-\frac{n}{2}}\int_{sM+y} e^{-\frac{|x|^2}{4}}, \]
i.e.~the supremum of the Gaussian integral over all dilations and translations of the submanifold. See \cite{CM-generic} for further discussion. Colding-Minicozzi prove that an ancient mean curvature flow $M_t^n\subset \R^N$ must lie in a Euclidean subspace of dimension $d\leq C_n\sup_{t}\lambda(M_t)$ (see \cite[Corollary 0.7]{CM-complexity}). Our examples suggests that the sharp value of $C_1$ should be $2/\lambda(S^1)$. See the discussion of Conjecture \ref{conj:c1}.

Second, our examples illustrate an interesting relation between the codimension and the tangent flow. Recall that the tangent flow of an ancient mean curvature flow is the limit of $\frac{M_t}{\sqrt{-t}}$ as $t\to 0$ or $t\to -\infty$, and we call them the tangent flow at $0$ and $-\infty$ respectively. For the ancient solution we construct to prove Theorem \ref{thm:main}, the tangent flow at $0$ is the embedded circle, and the tangent flow at $-\infty$ is the circle with multiplicity. Therefore, our examples show that the tangent flow at $0$ \emph{cannot} bound the codimension of an ancient solution. However, our examples suggest that the tangent flow at $-\infty$ \emph{can} help to bound the codimension. In a paper in preparation \cite{SS}, we confirm this fact by generalizing a result of Colding-Minicozzi from \cite{CM-complexity}. 

\begin{remark}
After we uploaded the first version of this preprint on Arxiv, we notice that the main construction in this paper has been discovered by Altschuler D.-Altschuler S.-Angenent-Wu in \cite{AAAW}. In \cite{AAAW}, they constructed the ancient curve shortening flow from the symmetry of $R^N$.

Since the motivation of our preprint is different from \cite{AAAW}, we still hope the readers could find some innovating ideas from this preprint.
\end{remark}

\bigskip
{\bf Acknowledgements.} We would like to thank Professor Minicozzi for his suggestions and helpful comments. We would like to appreciate Christos Mantoulidis for bringing our attention to \cite{ChMa}, and explaining the main result to us. It is our pleasure to thank Dr.\ Gerovitch, Professor Jerison, and Professor Moitra for supporting our research.

\section{Construction}
In this section, we construct ancient mean curvature flows in Euclidean space with high codimension. We begin by constructing a family of solutions we call torus curves.

\subsection{Torus Curves} \label{sec:toruscurve}
Let $k_1,\ \hdots,\ k_m$ be an increasing list of positive integers. We construct a $t$-parametrized family of curves $\gamma^{(k_1, \hdots, k_m)}_t \subset \R^{2m}$ (we denote it by $\gamma_t$ when the integers $k_j$ are implied for ease of notation) with coordinate functions of the form
\begin{equation}\label{eq:toruscurve}
(\gamma_t(\theta))_{2j-1} = r(t)^{k_j^2}\cos(k_j\theta)~,~(\gamma_t(\theta))_{2j} = r(t)^{k_j^2}\sin(k_j\theta)
\end{equation}
for $j = 1, \hdots, m$ and $\theta \in [0, 2\pi)$, where $r(t)$ is a positive function.

Intuitively, $\gamma_t$ is a curve on the torus \[  S^1(r^{k_1^2}) \times \hdots \times  S^1(r^{k_m^2}) \subset \R^{2m} \]
that wraps around the $j$th copy of the circle $k_j$ times at constant speed.

The following proposition immediately implies Theorem \ref{thm:main}.

\begin{proposition}\label{prop:toruscurve}
There is a unique positive function $r(t)$ defined for all $t\in(-\infty, 0)$ satisfying
\begin{equation}\label{eq:toruslimit}
\lim_{t \to -\infty} r(t) = +\infty \text{~and~} \lim_{t \to 0} r(t) = 0
\end{equation}
so that the family of curves $\gamma_t$ with coordinate functions given by (\ref{eq:toruscurve}) defines an ancient solution to the curve shortening flow. In particular, $\gamma_t$ does not lie in any $(2m-1)$-dimensional Euclidean subspace.
\end{proposition}
\begin{proof}
To see that $\gamma_t$ does not lie in any $(2m-1)$-dimensional Euclidean subspace, we note that $\gamma_t$ has $2m$ linearly independent coordinate functions.

To show that $\gamma_t$ is an ancient solution to the curve shortening flow, we simply compute $\Delta_{\gamma_t}\gamma_t$ and $\partial_t \gamma_t$, and set them equal.

First, we compute $\Delta_{\gamma_t}\gamma_t$. The tangent velocity field of $\gamma_t$ is given by
\[ 
(\partial_{\theta}\gamma_t)_{2j-1} = -k_jr^{k_j^2}\sin(k_j\theta)~,~(\partial_{\theta}\gamma_t)_{2j} = k_jr^{k_j^2}\cos(k_j\theta)~,~|\partial_{\theta}\gamma_t|^2 = \sum_{j=1}^m k_j^2r^{2k_j^2}.
\]
Since $|\partial_{\theta} \gamma_t|$ is $\theta$-independent, the Laplacian $\Delta_{\gamma_t}$ is given by $|\partial_{\theta}\gamma_t|^{-2} \partial_{\theta}^2$.
Therefore, we have
\begin{equation}\label{eq:torusdelta}
(\Delta_{\gamma_t}\gamma_t)_{2j-1} = \frac{-k_j^2r^{k_j^2}\cos(k_j\theta)}{\sum_{j=1}^m k_j^2r^{2k_j^2}}~,~(\Delta_{\gamma_t}\gamma_t)_{2j} = \frac{-k_j^2r^{k_j^2}\sin(k_j\theta)}{\sum_{j=1}^m k_j^2r^{2k_j^2}}.
\end{equation}

Next, we compute $\partial_t\gamma_t$, which is given by
\begin{equation}\label{eq:torust}
(\partial_t\gamma_t)_{2j-1} = k_j^2r^{k_j^2-1}r'\cos(k_j\theta)~,~(\partial_t\gamma_t)_{2j} = k_j^2r^{k_j^2-1}r'\sin(k_j\theta).
\end{equation}

Combining (\ref{eq:torusdelta}) and (\ref{eq:torust}), we conclude that if $r$ satisfies the ODE
\begin{equation}\label{eq:ode}
r' = \frac{-r}{\sum_{j=1}^m k_j^2r^{2k_j^2}},
\end{equation}
then $\gamma_t$ is a curve shortening flow. So it suffices to find an ancient solution to the ODE (\ref{eq:ode}).

Via standard ODE techniques, one can show that there is a unique positive function $r(t)$ defined for $t \in (-\infty, 0)$ satisfying (\ref{eq:toruslimit}) and (\ref{eq:ode}). See Appendix Theorem \ref{thm:ODE theory} for detailed discussion.
\end{proof}

Before we consider the implications of the existence of the solution $\gamma_t$, we outline an important property it satisfies as $t$ tends to $-\infty$ and $0$.

\begin{proposition}\label{prop:torustangentflow}
The tangent flow to the solution $\gamma_t^{(k_1,\hdots,k_m)}$ at $t = -\infty$ is the multiplicity $k_m$ circle, and the tangent flow at $t = 0$ is the multiplicity $k_1$ circle.
\end{proposition}
\begin{proof}
Recall that to obtain the tangent flow at $-\infty$, we rescale the solution $\gamma_t$ by $\frac{1}{\sqrt{-t}}$ and take the limit as $t \to -\infty$. Since $r(t)$ satisfies the ODE (\ref{eq:ode}), we have by Theorem \ref{thm:ODE theory} that
\[\sum_{j=1}^m r^{2k_j^2}=-2t\ \text{for $t<0$}.\]
When $t<-m/2$ we have the $r>1$, which yields the inequality
\[
\frac{\sqrt{2}}{\sqrt{m}r^{k_m^2}} \leq \frac{1}{\sqrt{-t}} \leq \frac{\sqrt{2}}{r^{k_m^2}}.
\]
Since the first $2m-2$ components of $\gamma_t$ grow slower than $r^{k_m^2}$ and $r \to + \infty$ as $t \to -\infty$, the first $2m-2$ components of the rescaled flow become arbitrarily small as $t \to -\infty$. After rescaling, the last two components of $\gamma_t$ parametrize the multiplicity $k_m$ circle.

For the tangent flow at $0$, we make a similar argument. Since $r(t)$ satisfies (\ref{eq:ode}), it is straightforward to check that for $r \leq 1$ (i.e.~for $t\geq -m/2$ ), we have the bounds
\[
\frac{\sqrt{2}}{\sqrt{m}r^{k_1^2}} \leq \frac{1}{\sqrt{-t}} \leq \frac{\sqrt{2}}{r^{k_1^2}}.
\]
Since the last $2m-2$ components of $\gamma_t$ vanish faster than $r^{k_1^2}$ and $r \to 0$ as $t \to 0$, the last $2m-2$ components of the rescaled flow become arbitrarily small as $t \to 0$. After rescaling, the first two components of $\gamma_t$ parametrize the multiplicity $k_1$ circle.
\end{proof}

Using the limiting behavior of $\gamma_t$ as $t \to -\infty$, we can compute the entropy of $\gamma_t$.

\begin{corollary}\label{cor:torusentropy}
The curve shortening flow $\gamma_t^{(k_1,\hdots,k_m)}$ satisfies $\sup_t\lambda(\gamma_t) = k_m\lambda(S^1)$.
\end{corollary}
\begin{proof}
First, we bound $\sup_t \lambda(\gamma_t)$ from below. We compute
\[
\frac{1}{\sqrt{4\pi}} \int_{s\gamma_t} e^{-\frac{|x|^2}{4}} = s\sqrt{\pi} \exp\left(-\frac{s^2}{4}\sum_{j=1}^m r^{2k_j^2}\right)\sqrt{\sum_{j=1}^m k_j^2r^{2k_j^2}}.
\]
Setting $\tilde{s} = \sqrt{2}\left(\sum_{j=1}^m r^{2k_j^2}\right)^{-1/2}$, and recalling that $\lim_{t \to -\infty} r(t) = \infty$, we have
\[ 
\sup_t \lambda(\gamma_t) \geq \lim_{t \to -\infty} \frac{1}{\sqrt{4\pi}} \int_{\tilde{s}\gamma_t} e^{-\frac{|x|^2}{4}} = \lambda(S^1) \lim_{t \to -\infty} \frac{\sqrt{\sum_{j=1}^m k_j^2r^{2k_j^2}}}{\sqrt{\sum_{j=1}^m r^{2k_j^2}}} = k_m\lambda(S^1).
\]

Second, we bound $\sup_t \lambda(\gamma_t)$ from above. Let $\epsilon > 0$. By Proposition \ref{prop:toruscurve}, there is a $T_{\epsilon} < 0$ so that $r^{-1} < \epsilon$ and $r \geq 1$ for all $t \leq T_{\epsilon}$. We compute
\[ \begin{split}
\frac{1}{\sqrt{4\pi}}\int_{sr^{-k_m^2}\gamma_t + y} e^{-\frac{|x|^2}{4}}
& \leq \sqrt{k_m^2 + C\epsilon^2}\frac{s}{\sqrt{4\pi}}
\int_0^{2\pi}e^{-\frac{s^2}{4}((y_{2m-1}-\cos(k_m\theta))^2 + (y_{2m}-\sin(k_m\theta))^2)}d\theta\\
& = \sqrt{k_m^2 + C\epsilon^2}\frac{s}{\sqrt{4\pi}}
\int_0^{2\pi}e^{-\frac{s^2}{4}((y_{2m-1}-\cos(\theta))^2 + (y_{2m}-\sin(\theta))^2)}d\theta\\
& \leq \sqrt{k_m^2 + C\epsilon^2}\lambda(S^1),
\end{split} \]
where the second line follows from the change of variables $k_m\theta \mapsto \theta$ and the periodicity of the sinusoidal functions. By the monotonicity of entropy, we conclude that
\[ \sup_t \lambda(\gamma_t) = \lim_{t \to -\infty} \lambda(\gamma_t) \leq \lim_{\epsilon \to 0} \sqrt{k_m^2 + C\epsilon^2}\lambda(S^1) = k_m \lambda(S^1).
\]
Hence, we obtain the desired equality.
\end{proof}

Recall that in \cite[Corollary 0.6]{CM-complexity}, Colding-Minicozzi show that there are universal constants $C_n$, depending only on the intrinsic dimension $n$, so that if $M_t^n \subset \R^N$ is an ancient solution to the mean curvature flow, then $M_t$ lies in a Euclidean subspace of dimension $C_n\sup_t \lambda(M_t)$. As an initial application of the existence of the ancient solution $\gamma_t$, we use the entropy computation in Corollary \ref{cor:torusentropy} to bound the constant $C_1$ in this result. For the torus curve solution $\gamma_t^{(k_1, \hdots, k_m)}$, we obtain the bound $\lambda(S^1)C_1 \geq \frac{2m}{k_m}$, which is maximized when $k_m$ is as small as possible. Since $k_1,\ \hdots,\ k_m$ must be an increasing list of positive integers, we have $k_m \geq m$. Hence, the torus curve solution gives the bound $\lambda(S^1)C_1 \geq 2$, identical to the bound from the shrinking circle solution in $\R^2$.

Since the constant speed sinusoidal functions are natural choices for linearly independent functions with compact images, this example suggests the following conjecture for the sharp constant $C_1$ in \cite[Corollary 0.6]{CM-complexity}.

\begin{conjecture}\label{conj:c1}
The sharp value of $C_1$ is $\frac{2}{\lambda(S^1)}$. In particular, any ancient curve shortening flow solution $M_t^1 \in \R^N$ that does not lie in a lower dimensional Euclidean subspace satisfies
\[ \sup_t \lambda(M_t) \geq \frac{N}{2}\lambda(S^1). \]
\end{conjecture}

As a second application of the existence of the ancient solution $\gamma_t$, we note that the codimension of an ancient solution cannot be bounded by information about its tangent flow at $t=0$.

\begin{corollary}\label{cor:nobound}
For any integer $m$, there is an ancient curve shortening flow that does not lie in any $(2m-1)$-dimensional Euclidean subspace whose tangent flow at $t=0$ is the multiplicity 1 circle.
\end{corollary}
\begin{proof}
Take the torus curve solution $\gamma^{(k_1, \hdots, k_m)}$ with $k_1 = 1$ and apply Proposition \ref{prop:torustangentflow}.
\end{proof}

Despite the fact that information about an ancient solution as $t \to 0$ \emph{cannot} be used to bound its codimension, there is hope that information about the solution as $t \to -\infty$ \emph{can} bound the codimension. Due to a recent result of Colding-Minicozzi (see \cite[Theorem 0.9]{CM-complexity}), if the tangent flow to an $n$-dimensional ancient solution $M_t^n$ at $-\infty$ is a round cylinder, then $M_t^n$ lies in a subspace of dimension $n+1$. In \cite{SS}, motivated by the limiting behavior of the torus curve solution at $-\infty$, we prove a sharp codimension bound for ancient solutions that converge sufficiently rapidly to the multiplicity $m$ circle as $t \to -\infty$.

\subsection{Helices}
It is important to note that the torus curve construction only works in even dimensional Euclidean spaces. Here, we construct a similar class of ancient curve shortening flows that lie in odd dimensional Euclidean spaces.

Let $k_1,\ \hdots,\ k_m$ be an increasing list of positive integers. We construct a $t$-parametrized family of curves $\Gamma^{(k_1, \hdots, k_m)}_t \subset \R^{2m+1}$ (we denote it by $\Gamma_t$ when the integers $k_j$ are implied for ease of notation) with coordinate functions of the form
\begin{equation}\label{eq:helix}
(\Gamma_t(s))_{2j-1} = r(t)^{k_j^2}\cos(k_js)~,~(\Gamma_t(s))_{2j} = r(t)^{k_j^2}\sin(k_js)~,~(\Gamma_t(s))_{2m+1} = s
\end{equation}
for $j = 1, \hdots, m$ and $s \in \R$, where $r(t)$ is a positive function.

When $m=1$, the curve $\Gamma_t$ is a standard helix in $\R^3$. For larger values of $m$, the curve $\Gamma_t^{(k_1, \hdots, k_m)}$ generalizes the helix to higher codimension in the sense that the projection along the ``helical axis'' (i.e.~the $x_{2m+1}$-axis) is the compact torus curve $\gamma_{t'}^{(k_1,\hdots, k_m)}$ in lieu of a circle.

\begin{proposition}\label{prop:helix}
There is a positive function $r(t)$ defined for all $t\in\R$ satisfying 
\begin{equation}\label{eq:helixlimit}
\lim_{t \to -\infty} r(t) = +\infty \text{~and~} \lim_{t \to +\infty} r(t) = 0
\end{equation}
so that the family of curves $\Gamma_t$ with coordinate functions given by (\ref{eq:helix}) defines an ancient solution to the curve shortening flow. In particular, $\Gamma_t$ does not lie in any $2m$-dimensional Euclidean subspace.
\end{proposition}

The proof of Proposition \ref{prop:helix} precisely follows the strategy of Proposition \ref{prop:toruscurve}. For the details, see Theorem \ref{thm:ODE theory2}.

It is important to remark that, unlike the torus curve solution  $\gamma_t$, the helix solution $\Gamma_t$ is not only an ancient solution but an \emph{eternal} solution (i.e.~defined for all $t \in \R)$.

However, the helix solution is not compact, and as we show below, it has infinite entropy.

\begin{proposition}\label{prop:helixentropy}
The curve shortening flow $\Gamma_t^{(k_1,\hdots, k_m)}$ satisfies $\sup_t \lambda(\Gamma_t) = +\infty$.
\end{proposition}
\begin{proof}
Note that for $r \geq 1$ (i.e.~for $-t$ sufficiently large), the first $2m$ components of $r^{-k_m^2}\Gamma_t$ have absolute value at most $1$. The last component of $r^{-k_m^2}\Gamma_t$ is contained in $[-1, 1]$ for $s \in [-r^{k_m^2}, r^{k_m^2}]$. Then we compute
\[\begin{split}
\frac{1}{\sqrt{4\pi}} \int_{r^{-k_m^2}\Gamma_t} e^{-\frac{|x|^2}{4}} \geq \frac{1}{\sqrt{4\pi}}\int_{-r^{k_m^2}}^{r^{k_m^2}} \exp\left(-\frac{|r^{-k_m^2}\Gamma_t|^2}{4}\right)|\partial_s(r^{-k_m^2}\Gamma_t)|ds \geq \frac{2r^{k_m^2}}{\sqrt{4\pi}}e^{-\frac{2m+1}{4}}.
\end{split} \]
Then by the definition of entropy, we have
\[ \lambda(\Gamma_t) \geq Cr^{k_m^2}. \]
Since $r \to \infty$ as $t \to -\infty$ by Proposition \ref{prop:helix}, we conclude that $\sup_t\lambda(\Gamma_t) = +\infty$.
\end{proof}

\subsection{Product of Torus Curves}
In this subsection, we use the existence of the ancient torus curve solution to construct high codimension ancient solutions for the $n$-dimensional mean curvature flow. Since the Laplacian of a product manifold satisfies $\Delta_{M_1 \times M_2} = \Delta_{M_1} + \Delta_{M_2}$, the product of solutions to the mean curvature flow yields a new solution.

\begin{proof}[Proof of Theorem \ref{thm:compact}]
Take the product of $n$ appropriately chosen torus curves.
\end{proof}

It is straightforward to check that if $M_1 \subset \R^{N_1}$ and $M_2 \subset \R^{N_2}$ have finite entropy, then the product manifold $M_1 \times M_2 \subset \R^{N_1 + N_2}$ satisfies $\lambda(M_1 \times M_2) \leq \lambda(M_1)\lambda(M_2) < \infty$.
Hence, we can relax the compactness requirement to the requirement of finite entropy to obtain the following result.

\begin{theorem}\label{thm:finiteentropy}
Let $N$ and $n$ be positive integers with $N \geq n \geq 2$. There exists an $n$-dimensional, connected, ancient solution to the mean curvature flow in $\R^N$ with finite entropy that does not lie in any $(N-1)$-dimensional Euclidean subspace.
\end{theorem}
\begin{proof}
Recall that $\R^k$ is a static solution to the mean curvature flow with entropy 1. Hence, we can take the product of $\R^k$ with torus curves.
\end{proof}

\begin{remark}
Theorem \ref{thm:finiteentropy} can also be proved by the construction of Choi-Mantoulidis (see \cite{ChMa}). However, there are many interesting differences between these two constructions, capturing different features of the ancient mean curvature flow. We list some of them here.

\begin{enumerate}
    \item Our construction only works in $\R^N$; the Choi-Mantoulidis construction also works for arbitrary Riemannian manifolds.
    \item Our construction does not have a limit as $t\to-\infty$; the Choi-Mantoulidis construction converges to an unstable minimal submanifold as $t\to-\infty$.
    \item Our construction provides ancient curve shortening flows in $\R^N$; the Choi-Mantoulidis construction cannot construct ancient curve shortening flows in $\R^N$.
    \item Our construction has explicit tangent flows at $0$ and $-\infty$, and we can compute the entropy explicitly; the Choi-Mantoulidis construction does not carry such feature.
    \item Our construction gives compact ancient mean curvature flows in $\R^N$; the Choi-Mantoulidis construction can only produce non-compact ancient mean curvature flows in $\R^N$.
\end{enumerate}
\end{remark}

We conclude with a conjecture.

\begin{conjecture}
Let $N$ and $n$ be positive integers with $N \geq n$. There exists an $n$-dimensional, connected, ancient solution to the mean curvature flow in $\R^N$ with finite entropy that does not lie in any $(N-1)$-dimensional Euclidean subspace.
\end{conjecture}

Together, Theorems \ref{thm:main} and \ref{thm:finiteentropy} imply that there is always a finite entropy, connected, ancient solution of any codimension, \emph{except} in the case of even codimension for the curve shortening flow (i.e.~curves lying in odd dimensional Euclidean spaces). While elementary examples in this case remain elusive, we believe such solutions exist.

\appendix
\section{ODE Theory}

In this section, we prove the ODE result used in the construction of the ancient curve shortening flows above. We prove the solution $r(t)$ to the ODE (\ref{eq:ode}) exists uniquely.

\begin{theorem}\label{thm:ODE theory}
There exists a unique function $r(t)$ that satisfies the ODE
\[r'=\frac{-r}{\sum_{j=1}^m k_j^2 r^{2k_j^2}},\tag{$\star$}\]
and the asymptotic condition
\[\lim_{t\to-\infty}r(t)=+\infty,\ \lim_{t\to 0}r(t)=0,\]
\end{theorem}

\begin{proof}
Define 
\[F(r)=\frac{1}{2}\left(\sum_{j=1}^m  r^{2k_j^2}\right).\]
Then by taking differential we notice that if 
\[F(r(t))+t=C\]
for a constant $C$, then $r(t)$ is a solution to the ODE ($\star$). Thus it remains to show that there exists a unique $r(t)$ such that $F(r(t))+t$ is a constant and $r(t)$ satisfies the asymptotic behavior.

Since we require $\lim_{t\to 0}r(t)=0$, we must set the constant $C$ to be $0$. Then we can define $r(t)>0$ by solving the implicit equation
\[F(r(t))=-t\]
on the domain $t\in(-\infty,0)$. Since $F(r)$ is strictly monotone in $r$ when $r>0$, $r(t)$ is defined uniquely, and $r(t)$ is monotone decreasing. In particular for $t<-m/2$, $r(t)\geq 1$. Finally, since 
\[-t=F(r(t))\leq (m r(t)^{2k_m^2})/2,\] 
we have $r(t)\to \infty$ as $t\to-\infty$. Therefore this $r(t)$ has the required asymptotic behavior.
\end{proof}

Similarly we can prove the following existence and asymptotic behavior of solutions to Proposition \ref{prop:helix}.

\begin{theorem}\label{thm:ODE theory2}
There exists a family of functions $r(t)$ that satisfy the ODE
\[r'=\frac{-r}{1+\sum_{j=1}^m k_j^2 r^{2k_j^2}},\]
and the asymptotic condition
\[\lim_{t\to-\infty}r(t)=+\infty,\ \lim_{t\to \infty}r(t)=0.\]
These $r(t)$ satisfy the requirement of Proposition \ref{prop:helix}.
\end{theorem}

\begin{proof}
The proof is the same as the proof of Theorem \ref{thm:ODE theory}, except here we use the function
\[F(r)=\frac{1}{2}\left(2\log r(t)+\sum_{j=1}^m r^{2k_j^2}\right).\]

Then $F(r)=C-t$ for a constant $C$. Therefore for a fixed constant we can solve $r(t)$ for any $t\in(-\infty,\infty)$ satisfying $F(r)=C-t$. $F(r)$ is strictly monotone in $r$ when $r>0$, so $r(t)$ is defined uniquely for a given $C$, and monotone decreasing. In particular when $t<C-m/2$, we have $r(t)>1$ and
\[C-t=F(r(t))\leq m r^{2k_m^2}/2.\]
So $r(t)\to \infty$ as $t\to-\infty$. When $t\geq C-m/2$, we have $r(t)\leq 1$ and
\[C-t=F(r(t))\geq \log r(t).\]
So $r(t)\to 0$ as $t\to \infty$.
\end{proof}

\bibliography{bibfile}
\bibliographystyle{alpha}

\end{document}